\documentclass[11pt,reqno,british,a4paper,twoside]{amsart}

\usepackage[sc]{mathpazo}
\usepackage{amsmath,mathtools,amssymb,paralist,babel}
\usepackage{array}
\usepackage{booktabs,microtype}
\usepackage{caption,subcaption}
\usepackage{graphicx,youngtab}
\usepackage[capitalise,noabbrev]{cleveref}
\mathtoolsset{mathic=true}
\setdefaultenum{\upshape (i)}{}{}{}
\theoremstyle{plain}
\newtheorem{theorem}{Theorem}[section]
\newtheorem{proposition}[theorem]{Proposition}
\newtheorem{lemma}[theorem]{Lemma}
\newtheorem{corollary}[theorem]{Corollary}
\newtheorem{remark}[theorem]{Remark}

\numberwithin{equation}{section}
\numberwithin{table}{section}

\newcommand{\vph}{\vphantom{\int_p^{1^1}}}
\newcommand{\Lie}[1]{\operatorname{\rm{#1}}}
\newcommand{\lie}[1]{\operatorname{\mathfrak{#1}}}
\newcommand{\G}{\Lie{G}}
\newcommand{\LH}{\Lie{H}}
\newcommand{\LK}{\Lie{K}}
\newcommand{\LU}{\Lie{U}}
\newcommand{\GL}{\Lie{GL}}
\newcommand{\SO}{\Lie{SO}}
\newcommand{\SL}{\Lie{SL}}
\newcommand{\Sp}{\Lie{Sp}}
\newcommand{\SU}{\Lie{SU}}
\newcommand{\PSU}{\Lie{PSU}}
\newcommand{\Gr}{{\Lie{Gr}}}
\newcommand{\g}{\lie{g}}
\newcommand{\lk}{\lie{k}}
\newcommand{\lh}{\lie{h}}
\newcommand{\gl}{\lie{gl}}
\newcommand{\so}{\lie{so}}
\newcommand{\su}{\lie{su}}
\newcommand{\lu}{\lie{u}}
\newcommand{\lsp}{\lie{sp}}

\newcommand{\bC}{{\mathbb C}}
\newcommand{\bH}{{\mathbb H}}
\newcommand{\bR}{{\mathbb R}}
\newcommand{\pC}{\mathbb{CP}}
\newcommand{\pH}{\mathbb{HP}}

\newcommand{\Ld}{\mathcal L}

\newcolumntype{C}{>{$}c<{$}}

\DeclareMathOperator{\im}{im}
\DeclareMathOperator{\Ad}{Ad}

\DeclareMathOperator{\End}{End}

\newcommand{\hook}{{\lrcorner\,}}

\DeclarePairedDelimiter{\Span}{\langle}{\rangle}
\newcommand{\abs}[1]{\left\lvert#1\right\rvert}

\newcommand{\D}[2]{\frac{\partial #1}{\partial #2}}

\newcommand{\norm}[1]{\left\lVert#1\right\rVert}

\begin{document}

\title[\( 8 \)-dimensional quaternionic geometry]
{Quaternionic geometry in dimension eight}

\author{Diego Conti}
\address[D.~Conti]{ Dipartimento di Matematica e Applicazioni\\ Universit\`a di Milano
  Bicocca\\ Via Cozzi 55\\ 20125 Milano\\ Italy.}
\email{diego.conti@unimib.it}

\author{Thomas Bruun Madsen}
\address[T.\,B.~Madsen]{Department of Mathematics\\ Aarhus University\\
Ny Munkegade 118, Bldg 1530\\ 8000 Aarhus\\ Denmark.}
\email{thomas.madsen@math.au.dk}

\author{Simon Salamon}
\address[S. Salamon]{Department of Mathematics\\ King's College London\\
Strand\\  London WC2R 2LS\\ United Kingdom.}
\email{simon.salamon@kcl.ac.uk}

\begin{abstract}
  We describe the \( 8 \)-dimensional Wolf spaces as cohomogeneity one
  \( \SU(3) \)-manifolds, and discover perturbations of the
  quaternion-k\"ahler metric on the simply-connected \( 8 \)-manifold \(
  \G_2\!/\!\SO(4) \) that carry a closed fundamental \( 4 \)-form but are
  not Einstein.
\end{abstract}

\maketitle

\centerline{\it To Nigel Hitchin on the occasion of his 70th birthday}

\bigskip

\section{Introduction}
\label{sec:intro}

Of the ``fundamental geometries'' captured by Berger's list of
holonomy groups, the quaternionic unitary group stands out in that
Riemannian manifolds with holonomy in \( \Sp(n)\Sp(1) \) are Einstein
but not Ricci-flat, unless locally hyperk\"ahler. Excluding the latter
case, the study of these quaternion-k\"ahler manifolds splits into two
cases, depending on the sign of the scalar curvature. The negative
case is fairly flexible
\cite{Alekseevski:qK1,Alekseevski:qK2,Cortes:qKASp,Cortes:non-hom-qK},
but the situation of positive scalar curvature is extremely rigid. In
fact, it is conjectured that a complete positive quaternion-k\"ahler
manifold is necessarily one of the symmetric spaces that were first
described by Wolf \cite{Wolf:hom_qK}. This rigidity suggests a quest
for ways of weakening the holonomy condition.  It turns out that
dimension \( 8 \) harbours a particularly natural type of almost
quaternion-k\"ahler manifold.

An appealing way of expressing an almost quaternionic Hermitian structure is to say that
our \( 8 \)-manifold admits a \( 4 \)-form that is pointwise linearly equivalent to
\begin{equation}
\label{eq:quat_4_form}
\Omega=\tfrac12(\omega_1^2+\omega_2^2+\omega_3^2),
\end{equation}
where \( (\omega_1,\omega_2,\omega_3) \) is the standard hyperk\"ahler triplet on \( \bR^8\cong\bH^2 \):
\begin{equation*}
\begin{cases}
\omega_1=dx^{12}+dx^{34}+dx^{56}+dx^{78},\\
\omega_2=dx^{13}+dx^{42}+dx^{57}+dx^{86},\\
\omega_3=dx^{14}+dx^{23}+dx^{58}+dx^{67}.
\end{cases}
\end{equation*}
In these terms, the quaternion-k\"ahler condition then amounts to \( \Omega \) being parallel for the Levi-Civita connection,
\( \nabla\Omega=0 \). 

Swann \cite{Swann:AspectsSymplectiquesDe} observed that it is possible
to have \( \Omega \) non-parallel and closed (and so harmonic), but
that closedness of the fundamental \( 4 \)-form implies
quaternion-k\"ahler in dimension at least \( 12 \).  Using exterior
differential systems, Bryant \cite{Bryant:closedsp2sp1} analysed the
local existence of the ``harmonic'' \( \Sp(2)\Sp(1) \)-structures and
showed that solutions exist in abundance, even though the PDE system
at first looks overdetermined. In fact, involutivity of the exterior
differential system can be deduced by observing that the contraction
of \( \Omega \) with any vector \( v\in \bR^8 \) induces a stable form
on the quotient \( \bR^8\!/\!\Span{v} \), cf.\
\cite{Conti:EmbeddingIntoManifolds,Conti-M:Harmonic}. In
\cite{Salamon:Almostparallelstructures} the third author provided the
first compact non-parallel example of such a geometry. Later many more
examples have followed
\cite{Giovannini:SpecialStructuresAnd,Conti-M:Harmonic} by reducing
the internal symmetry group from the quaternionic unitary group to its
intersection with \( \SO(6) \) and \( \SO(7) \).

These previously known examples of non-parallel harmonic \( \Sp(2)\Sp(1) \)-manifolds all have infinite fundamental group and associated metric 
of negative scalar curvature. A natural question is whether such structures with positive scalar curvature exist on simply-connected manifolds.

\subsection*{Acknowledgements.}

We all thank Robert Bryant, David Calderbank and Andrew Swann for
useful comments. DC was partially supported by FIRB 2012 ``Geometria
differenziale e teoria geometrica delle funzioni''. TBM gratefully
acknowledges financial support from Villum Fonden. The completion of
this work was supported by a grant from the Simons Foundation
(\#488635, Simon Salamon), which will also provide a forum for pursuing
topics mentioned in the final section. The interest of all three
authors in special holonomy can be traced back to Nigel Hitchin, who
recognised long ago the importance of this field.

\section{The Wolf spaces}
\label{sec:wolf-spaces}

In Wolf's construction \cite{Wolf:hom_qK} one starts with a compact centreless simple Lie group \( \G \) with Lie algebra \( \g \) and choice of Cartan subalgebra \( \lie{t}\subset\g \). One then picks a maximal root \( \beta\in \lie{t} \) and considers an associated \( \lsp(1) \) and its centraliser \( \lie{l}_1 \) in \( \g \). The Lie algebra \( \lk=\lsp(1)\oplus\lie{l}_1 \) will model the holonomy algebra of a symmetric space: if we let \( \G \) be the simply-connected compact simple Lie group corresponding to \( \g \) and \( \LK \) the compact subgroup generated by \( \lk \), then \( \G\!/\! \LK \) is a compact symmetric quaternion-k\"ahler manifold, a so-called \emph{Wolf space}, with holonomy \( \LK \). The associated quaternionic structure on the tangent space is generated by the subgroup \( \Sp(1)\subset\LK \).

\subsection{Quaternionic projective plane}
\label{sec:HP2}

The model space for a quaternion-k\"ahler \( 8 \)-manifold is the quaternionic projective plane
\begin{equation*}
\pH(2)=\frac{\Sp(3)}{\Sp(2)\times\Sp(1)}. 
\end{equation*}
In Wolf's terms, we can describe this as follows.
 
Choose the Cartan subalgebra \( \lh_{\bC} \) of \( \lsp(3)_{\bC} \) spanned by the three elements
\begin{equation*}
H_j=E_{i,i}-E_{i+3,i+3},
\end{equation*}
where the matrix \( E_{k,\ell} \) has only non-zero entry, equal to \( 1 \), at position \( (k,\ell) \). Then let \( L_j\in \lie{h}^* \) be the element satisfying \( L_j(H_i)=\delta_{ji} \). The corresponding roots of  \( \lsp(3)_{\bC} \) are the vectors \( \pm L_i\pm L_j \). The associated eigenspaces are spanned by
\begin{equation*}
\begin{gathered}
X_{i,j}=E_{i,j}-E_{3+j,3+i},\\
Y_{i,j}=E_{i,3+j}+E_{j,3+i},\,Z_{i,j}=E_{3+i,j}+E_{3+j,i},\\
U_i=E_{i,3+i},\, V_i=E_{3+i,i},
\end{gathered}
\end{equation*}
where \( i\neq j \) in the first two rows.

A real structure \( \sigma \) is determined by
\begin{equation*}
\begin{gathered}
\sigma(H_i)=-H_i,\, \sigma(X_{i,j})=-X_{j,i},\\
\sigma(Y_{i,j})=-Z_{i,j},\,\sigma(Z_{i,j})=-Y_{i,j},\\ 
\sigma(U_i)=-V_i,\,\sigma(V_i)=-U_i,
\end{gathered}
\end{equation*}
and we can therefore choose a basis of \( \lsp(3) \) given by
\begin{equation*}
\{\underbrace{iH_k}_{A_k},\overbrace{X_{k,\ell}\!-\!X_{\ell,k}}^{P_{k+\ell-2}},\underbrace{i(X_{k,\ell}\!+\!X_{\ell,k})}_{P_{k+\ell+1}},\overbrace{Y_{k,\ell}\!-\!Z_{k,\ell}}^{Q_{k+\ell-2}},\underbrace{i(Y_{k,\ell}\!+\!Z_{k,\ell})}_{Q_{k+\ell+1}},\overbrace{U_k\!-\!V_k}^{R_k},\underbrace{i(U_k\!+\!V_k)}_{R_{k+3}}\}. 
\end{equation*}

In these terms Wolf's highest root \( \lsp(1) \) is given by
\begin{equation*}
\lsp(1)=\Span{A_1,R_1,R_4},
\end{equation*}
and its centraliser is the copy of \( \lsp(2) \) described as
\begin{equation*}
\lie{k}_1=\Span{A_2,A_3,P_3,P_6,Q_3,Q_6,R_2,R_3,R_5,R_6}.
\end{equation*}

We now have the direct sum decomposition
\begin{equation*}
\lsp(3)=\lie{k}\oplus\lie{p},
\end{equation*}
where \( \lie{k}=\lsp(1)\oplus\lsp(2) \) and
\begin{equation*}
\lie{p}=\lie{k}^\perp=\Span{P_1,P_2,P_4,P_5,Q_1,Q_2,Q_4,Q_5}.
\end{equation*}
The basis
\begin{equation*}
\tfrac1{4\sqrt2}P_1,\,\tfrac1{4\sqrt2}P_4,\,\tfrac1{4\sqrt2}Q_1,\,\tfrac1{4\sqrt2}Q_4,\,\tfrac1{4\sqrt2}P_2,\,\tfrac1{4\sqrt2}P_5,\,\tfrac1{4\sqrt2}Q_2,\,\tfrac1{4\sqrt2}Q_5,
\end{equation*}
is orthonormal for the Killing metric on the subspace \( \lie{p}=(\lsp(1)\oplus\lsp(2))^\perp \), 
and determines an adapted frame for the \( \Sp(2)\Sp(1) \)-structure. Its dual basis is given by
\begin{equation*}
\begin{gathered}
f^1=4\sqrt2p_1,\, f^2=4\sqrt2p_4,\, f^3=4\sqrt2q_1,\,f^4=4\sqrt2q_4,\, f^5=4\sqrt2p_2,\,\\ f^6=4\sqrt2p_5,\, f^7=4\sqrt2q_2,\,f^8=4\sqrt2q_5, 
\end{gathered}
\end{equation*}
where \( p_1,\dots, q_5\) is the dual basis of \( P_1,\ldots, Q_5 \).

\subsection{Complex Grassmannian}
\label{sec:Gr2-C4}

Consider next the complex Grassmannian of planes in \( \bC^4 \):
\begin{equation*}
\Gr_2(\bC^4)=\frac{\SU(4)}{\Lie{S}(\Lie{U}(2)\times\Lie{U}(2))}.
\end{equation*} 

In order to describe Wolf's structure on this space, 
we begin by considering \( \SU(4)_{\bC}=\SL(4,\bC) \) with its usual basis:
\begin{equation*}
H_i\!=\!E_{i,i}-E_{i+1,i+1},\, X_1\!=\!E_{1,2},X_2\!=\!E_{1,3},\ldots,X_6\!=\!E_{3,4},\, Y_1\!=\!E_{2,1},\ldots, Y_6\!=\!E_{4,3}.
\end{equation*}
From the real structure \( \sigma \), given by \( \sigma(H_i)=-H_i \), \( \sigma(X_i)=-Y_i \), \( \sigma(Y_i)=-X_i \), we see that a basis of \( \su(4) \) can be described as
\begin{equation*}
\su(4)=\Span{
\underbrace{iH_j}_{A_j},\underbrace{X_j-Y_j}_{C_j},\underbrace{i(X_j+Y_j)}_{B_j}}.
\end{equation*}

In these terms, Wolf's highest root \( \lsp(1) \) reads
\begin{equation*}
\lsp(1)=\Span{i(H_1+H_2+H_3),C_3,B_3}
\end{equation*}
and its centraliser is
\begin{equation*}
\lie{k}_1=\Span{i(H_1-H_3),iH_2,C_4,B_4}.
\end{equation*}

As a result we have the direct sum decomposition
\begin{equation*}
\su(4)=\lie{k}\oplus \lie{p},
\end{equation*}
where \( \lie{k}=\lsp(1)\oplus\lie{k}_1 \) and
\begin{equation*}
\lie{p}=\lie{k}^\perp=\Span{C_1,B_1,C_2,B_2,C_5,B_5,C_6,B_6}.
\end{equation*}
We see that the orthonormal frame
\begin{equation*}
\tfrac14C_1,\tfrac14B_2,\tfrac14C_5,\tfrac14B_5,\tfrac14C_2,\tfrac14B_2,\tfrac14C_6,\tfrac14B_6
\end{equation*}
is adapted to the \( \Sp(2)\Sp(1) \)-structure. Letting \( c_1,\ldots, b_6 \) denote the dual basis
of \( C_1,\ldots,B_6 \), we then have an \( \Sp(2)\Sp(1) \)-adapted coframe \( f^1,\ldots, f^8 \) given by \( f^1=4c_1 \), and so forth.

\subsection{The exceptional Wolf space}
\label{sec:exc-Wolf-space}

We finally turn to Wolf's construction of a quaternion-k\"ahler structure on 
\begin{equation*}
\frac{\G_2}{\SO(4)}.
\end{equation*}
This is slightly more involved, due to the more complicated nature of \( \g_2 \). 

First we need to choose a suitable basis of \( (\g_2)_{\bC} \). We shall follow \cite{Fulton-H:Rep}, where the long roots are given by
\begin{equation*}
\alpha_2=(-\tfrac32,\tfrac{\sqrt3}2)=-\beta_2, \alpha_5=(\tfrac32,\tfrac{\sqrt3}2)=-\beta_5,\alpha_6=(0,\sqrt3)=-\beta_6,
\end{equation*}
and the short roots are
\begin{equation*} 
\alpha_1=(1,0)=-\beta_1, \alpha_3=(-\tfrac12,\tfrac{\sqrt3}2)=-\beta_3, \alpha_4=(\tfrac12,\tfrac{\sqrt3}2)=-\beta_4. 
\end{equation*} 
We shall pick \( \beta=(0,\sqrt3) \) as the highest root.

The real structure \( \sigma \) of \( \g_2^\bC \) is determined by 
\begin{equation*}
\sigma(H_i)=-H_i, \quad \sigma(X_i)=-Y_i, \quad \sigma(Y_i)=-X_i,
\end{equation*}
and a basis of \( \g_2 \) is therefore given by
\begin{equation*}
\{ A_1=iH_1,A_2=iH_2, W_j=X_j-Y_j, Z_j=i(X_j+Y_j)\colon\, 1\leqslant j\leqslant 6 \}.
\end{equation*}

The highest root \( \lsp(1) \) is given by
\begin{equation*}
\lsp(1)=\Span{A_1+2A_2, W_6,Z_6},
\end{equation*}
and its centraliser in \( \g_2 \) is given by
\begin{equation*}
\lie{l}_1=\Span{A_1, W_1, Z_1}.
\end{equation*}

We now have the direct sum decomposition
\begin{equation*}
\g_2=\lk\oplus\lie{p},
\end{equation*}
where 
\begin{equation*}
\lie{p}=\lk^\perp =\Span{W_2,W_3,W_4,W_5, Z_2,Z_3,Z_4,Z_5}.
\end{equation*}
In particular, the basis
\begin{equation*}
\begin{gathered}
W_2,Z_2, -W_5, -Z_5, \tfrac1{\sqrt3}W_3,\tfrac{1}{\sqrt3}Z_3,-\tfrac1{\sqrt3}W_4,-\tfrac{1}{\sqrt3}Z_4, 
\end{gathered}
\end{equation*}
which is orthonormal for the Killing form on the subspace \( \lie{p}=\so(4)^\perp \), determines an adapted frame for this \( \Sp(2)\Sp(1) \)-structure.
As in the previous cases, we let \( f^i \) denote its dual coframe, meaning \( f^1=w_2 \), etc.

\section{Cohomogeneity one \( \SU(3) \)-actions}
\label{sec:cohom-one-ac}

Each of the \( 8 \)-dimensional Wolf spaces \( M \) admits a cohomogeneity one \( \SU(3) \)-action, which was studied by Gambioli \cite{Gambioli:SU3action},
 see also \cite{Kuroki:SU3}. We summarise and elaborate on key facts below. 
In each case, the action comes from the embedding \( \SU(3)\subset \G \).
In order to explicitly describe the orbits of this \( \SU(3) \)-action, we choose an element
\( Z\in\mathfrak{p}\cap \su(3)^\perp \) and write \( \gamma(t)=\exp(tZ) \). 
Then the \( \SU(3) \)-orbits of \( \gamma(t) \) are given by
\begin{equation*}
\iota_t\colon \SU(3)\to M, \quad g\mapsto g\gamma(t)\Lie{K}.
\end{equation*}

As the quotient map
\begin{equation*}
\pi\colon G\to M, \quad g \mapsto g\Lie{K}
\end{equation*}
has \( \ker\pi_{*g}=L_{g*}(\lie{k}) \), we can identify \( T_{\gamma(t)\Lie{K}}M\) with \( L_{\gamma(t)*}(\mathfrak{p}) \), or simply \( \mathfrak{p} \) where left translation is then understood. It follows that we can identify \( \iota_{t*} \) with the map \( \su(3)\to \lie{p} \) given by
\begin{equation*}
X\mapsto[\Ad(\gamma(t)^{-1})(X)]_{\lie{p}}.
\end{equation*}
Since \( Z \) is orthogonal to \( \su(3)\subset\lie{g} \), it is clear from invariance of the Killing form that \( L_{\gamma(t)*}(Z) \) is orthogonal to the \( \SU(3) \)-orbit of \( \gamma(t) \) for all \( t \),
and altogether the cohomogeneity one action infinitesimally is described by the mapping
\begin{equation}
\label{eq:orbmap}
\su(3)\oplus \bR \to \lie{p},\quad X\mapsto [\Ad(\gamma(t)^{-1})(X)]_{\mathfrak{p}},\, \D{}{t}\mapsto Z.
\end{equation}

Now, given the adapted quaternion-k\"ahler frame on \( T_{e\LK}M \), we can use \eqref{eq:orbmap} 
to pull this back to \( \su(3)\oplus\bR \) and thereby get a description of the Wolf space structure
that is adapted to the cohomogeneity one setting.

Before doing so, we fix some conventions for \( \su(3) \): in the following \( e^1,\ldots, e^8 \) will always
denote a basis of \( \su(3)^* \) such that the following structure equations hold:
\begin{equation}
\label{eq:su3}
\begin{gathered}
de^1=-e^{23}-e^{45}+2 e^{67},\quad de^2=e^{13}+e^{46}-e^{57}-\sqrt{3} e^{58},\\
de^3=-e^{12}-e^{47}+\sqrt{3}e^{48}-e^{56},\quad de^4=e^{15}-e^{26}+e^{37}-\sqrt{3}e^{38},\\
de^5=-e^{14}+e^{27}+\sqrt{3} e^{28}+e^{36},\quad de^6=-2 e^{17}+e^{24}-e^{35},\\
de^7=2 e^{16}-e^{25}-e^{34},\quad de^8=- \sqrt{3}(e^{25}-e^{34}).
\end{gathered}
\end{equation}
In terms of matrices, we can express the dual basis \( e_1,\dotsc, e_8 \) as
\begin{equation*}
\begin{gathered}
e_1=E_{21}-E_{12},\, e_2=E_{31}-E_{13},\, e_3=E_{32}-E_{23},\\
e_4=-i(E_{23}+E_{32}),\, e_5=i(E_{13}+E_{31}),\, e_6=-i(E_{12}+E_{21}),\\
e_7=i(E_{22}-E_{11}),\, e_8=\tfrac i{\sqrt3}(2E_{33}-E_{11}-E_{22}).
\end{gathered}
\end{equation*}

In the following subsections, we shall show that, from the cohomogeneity one \( \SU(3) \) point
of view, Wolf's quaternion-k\"ahler manifolds arise by combining three basic models that
correspond to tubular neighbourhoods \( \G\times_{\LH}V \) 
of the relevant singular orbits \( \G\!/\!\LH \). These basic building blocks are
summarised in \cref{tab:models}, where \( \Sigma^2 \) is the irreducible \( 3 \)-dimensional 
representation of \( \SU(2) \), \( K=\Lambda^{2,0} \) denotes the \( 1 \)-dimensional representation of \( \LU(2) \)
corresponding to the determinant, and \( \bR^3 \) and \(\bC^2 \) are the standard representations
of \( \SO(3) \) and \( \LU(2) \), respectively.

\begin{table}[htp]
  \centering
  \begin{tabular}{C|C|C|C}
    \toprule
    \LH & \lie g/\lie h & V & \G/\LH\\
   \hline\vph
    \SU(2) & \bR\oplus \bH & \Sigma^2 & S^5\\
\hline\vph
\SO(3) & \odot^2_0\bR^3 & \bR^3 & L\\
\hline\vph
\LU(2) & [\![\Lambda^{1,0}K]\!] & \bC^2 & \pC(2)\\
    \bottomrule
  \end{tabular}
\vskip5pt
  \caption{The three Wolf space building blocks.}
  \label{tab:models}
\end{table}
\vskip-5pt

In \cref{tab:models}, \( L \) is the symmetric space \( \SU(3)\!/\!\SO(3) \) that parametrises special Lagrangian 
subspaces of \( \bR^6\cong\bC^3 \).

\subsection{Quaternionic projective plane}
\label{sec:SU3-HP2}

In order to give a cohomogeneity one description of the quaternionic projective plane, we 
start by fixing the embedding of \( \SU(3) \) in \( \Lie{Sp}(3) \) given via
\begin{equation*}
\SU(3)\subset\left\{\left(\begin{smallmatrix}X & 0 \\ 0 & (X^{-1})^T \end{smallmatrix}\right)\colon\, X\in \GL(3,\bC) \right\}\subset \Sp(3,\bC).
\end{equation*}
Correspondingly, we have the following description of \( \su(3) \) at the Lie algebra level:
\begin{equation*}
\su(3)=\Span{A_1-A_2,A_2-A_3,P_1,P_2,P_3,P_4,P_5,P_6}.
\end{equation*}
Now choose an element
\begin{equation*}
Z\in \lie{p}\cap\su(3)^{\perp}=\Span{Q_1,Q_2,Q_4,Q_5}.
\end{equation*}
We shall fix \( Z=Q_1 \). In these terms, the adapted quaternion-k\"ahler frame pulls back 
as follows.

\begin{lemma}
\label{lem:pull-back-adap-frame-HP2}
Under the mapping \eqref{eq:orbmap}, Wolf's coframe pulls back to the dual of \( \su(3)\oplus\bR \) to give
\begin{equation}
\label{eq:tdep_coframeHP2}
\begin{gathered}
\tilde e^1(t)=4\sqrt{2}\cos(2t) e^6,\, \tilde e^2(t)=-4\sqrt{2}\cos(2t) e^7,\\
\tilde e^3(t)=4\sqrt{2} dt,\, \tilde e^4=\tfrac{4\sqrt{6}}3\sin(2t) e^8,\\
\tilde e^5(t)=4\cos(t)(e^2+e^4),\, \tilde e^6(t)=4\cos(t)(e^3+e^5),\\
\tilde e^7(t)=4\sin(t)(e^2-e^4),\, \tilde e^8(t)=4\sin(t)(e^3-e^5).
\end{gathered}
\end{equation}
\end{lemma}

\begin{proof}
In order to write things consistently with the structure equations \eqref{eq:su3},
we fix on \( \su(3)\subset\lsp(3) \) the basis
\begin{equation*}
\begin{gathered}
e_1=A_1-A_2,\, e_2=\tfrac1{\sqrt2}(P_2-P_3),\, e_3=\tfrac1{\sqrt2}(P_5+P_6),\,e_4=\tfrac1{\sqrt2}(P_2+P_3),\\
e_5=\tfrac1{\sqrt2}(P_5-P_6),\, e_6=P_1,\,e_7=-P_4,\, e_8=\tfrac1{\sqrt3}(A_1+A_2-2A_3).
\end{gathered}
\end{equation*}

Computing the action of \( \Ad(\exp (-tQ_1)) \) with respect to the bases 
\begin{equation*}
e_1,\dotsc, e_8, \quad P_1,P_4,Q_1,Q_4, P_2,P_5,Q_2,Q_5
\end{equation*}
of \( \su(3) \) and \( \lie{p}\), respectively, we find that the map \eqref{eq:orbmap} is represented by the matrix 
\begin{equation*}
 \left(
\begin{smallmatrix}
0&0&0&0&0&\cos(2t)&0&0\\0&0&0&0&0&0&- \cos(2t)&0\\0&0&0&0&0&0&0&0\\0&0&0&0&0&0&0&\frac{1}{3}\sqrt{3} \sin(2t)\\0&\frac{1}{2}\cos(t) \sqrt{2}&0&\frac{1}{2}\cos(t) \sqrt{2}&0&0&0&0\\0&0&\frac{1}{2}\cos(t) \sqrt{2}&0&\frac{1}{2}\cos(t) \sqrt{2}&0&0&0\\0&\frac{1}{2}\sin(t) \sqrt{2}&0&-\frac{1}{2}\sin(t) \sqrt{2}&0&0&0&0\\0&0&\frac{1}{2}\sin(t) \sqrt{2}&0&-\frac{1}{2}\sin(t) \sqrt{2}&0&0&0
\end{smallmatrix}
\right).
\end{equation*}
Considering its transpose, and recalling that \( \D{}{t} \) maps to \(
Q_1 \), we see that the adapted coframe \( f^1,\ldots, f^8 \) pulls
back to \( \tilde e^1(t),\dots, \tilde e^8(t) \) to give the stated
result.
\end{proof}

Our computation confirms, more directly, the following result that is indicated in \cite{Gambioli:SU3action}.

\begin{proposition}
\label{prop:cohom1-pH2}
The Wolf space \( \pH(2) \) can be viewed as a cohomogeneity one manifold obtained by gluing together the disc bundles over the singular orbits \( \pC(2)=\SU(3)\!/\!\Lie{U}(2) \) and \( S^5=\SU(3)\!/\!\SU(2) \). Each principal orbit is a copy of the exceptional Aloff-Wallach space \( N^{1,-1}\cong N^{1,0} \).
\end{proposition} 

\begin{proof}
The point is to identify the principal and singular stabilisers; the latter appear at \( t=0 \) and \( t=\pi/4 \).
Since the singular orbits have codimension strictly smaller than \( 6 \), both singular stabilisers
are connected \cite[Corollary 1.9]{Hoelscher:cohom1}. In particular, it suffices to work at the Lie algebra level, since connected subgroups of
\( \SU(3) \) are in one-to-one correspondence with subalgebras of \( \su(3) \). 
It is worthwhile making this more explicit by identifying the Lie algebras of the principal and singular stabilisers.

Regarding the principal orbits, we observe that the coframe
\eqref{eq:tdep_coframeHP2} for generic \( t \) annihilates the \(
\lu(1) \) spanned by \( e_1 \).  When \( t=0 \), the coframe is the
annihilator of a Lie algebra \( \lu(2) \) spanned by the four elements \(
e_1, e_8 \), \( e_2-e_4, e_3-e_5 \).  Finally, at \( t=\pi/4 \) the subspace
annihilated is the \( \su(2) \) spanned by \( e_1, e_6, e_7 \).
\end{proof}

For later reference, let us emphasise that the tangent space of the open set corresponding to the principal orbits 
at each point decomposes as the \( \LU(1) \)-representation 
\begin{equation}
\label{eq:tspace-prc-orb}
\bR^8\cong2\bR\oplus 2V_1\oplus V_2,
\end{equation}
where \( V_k \) is the irreducible \( 2 \)-dimensional representation on which the principal \( \LU(1) \) acts via matrices 
of the form
\begin{equation*}
\left(
\begin{smallmatrix}
\cos(k\theta) & \sin(k\theta)\\
-\sin(k\theta) & \cos(k\theta) 
\end{smallmatrix}
\right);
\end{equation*}
at the infinitesimal level, this follows directly from \eqref{eq:su3}.

As a final remark, note that \( \pH(2) \) comes with a \( \LU(1) \)-action, generated by the diagonal \( \LU(1) \) in \( \LU(3)\subset\Sp(3) \), that 
commutes with the action of \( \SU(3) \). Clearly, this circle action generates a Killing vector field. Explicitly, this action is generated by \( X=\sqrt3e_8 + 3A_3 \) and 
in our cohomogeneity one framework, it reads
\begin{equation*}
\begin{split}
\iota_t(g)\mapsto \exp(sX)\iota_t(g) &= g\exp(sX)\gamma(t)K= g\exp(\sqrt3 s e_8)\exp(3sA_3)\gamma(t)\LK\\
&= g\exp(\sqrt3 s e_8)\gamma(t)\exp(3sA_3)\LK= \iota_t(g\exp(\sqrt3 s e_8)),
\end{split}
\end{equation*}
where we have used the fact that \( A_3 \) is an element of \( \lsp(2)\subset\lsp(2)\oplus\lsp(1) \) that commutes with \( Z=Q_1 \). Therefore, the Killing vector field \( X \) can be identified with the left-invariant vector field \( \sqrt3 e_8 \) on \( \SU(3) \).

\subsection{Complex Grassmannian}
\label{sec:SU3-Gr2-C4}

In order to get an explicit description of the cohomogeneity one nature of \( \Gr_2(\bC^4) \), we fix the copy
of \( \SU(3) \) which comes from the usual embedding 
\begin{equation*}
 \SU(3)\cong\left\{\left(\begin{smallmatrix} A & \\ & 1\end{smallmatrix}\right)\colon A\in\SU(3)\right\}\subset\SU(4) .
\end{equation*} 
At the level of Lie algebras, this means that we are working with the copy \( \su(3)\subset\su(4) \) spanned by
\begin{equation*}
\begin{gathered}
e_1=A_1+A_2, e_2=\frac1{\sqrt2}(C_1+C_4),e_3=\frac1{\sqrt2}(B_1+B_4),e_4=\frac1{\sqrt2}(C_1-C_4),\\
e_5=\frac1{\sqrt2}(B_1-B_4),e_6=C_2,e_7=-B_2,e_8=\frac1{\sqrt3}(A_1-A_2);
\end{gathered}
\end{equation*}
this choice of a basis is consistent with \eqref{eq:su3}.
Proceeding as before, we pick an element
\begin{equation*}
Z\in\lie{p}\cap\su(3)^{\perp}=\Span{C_5,B_5,C_6,B_6},
\end{equation*}
which, for concreteness, we shall fix to \( Z=C_5 \). 
Computations as in the proof of \cref{lem:pull-back-adap-frame-HP2}
then give:

\begin{lemma}
Under the mapping \eqref{eq:orbmap}, Wolf's coframe pulls back to the dual of \( \su(3)\oplus\bR \) to give
\begin{equation*}
\begin{gathered}
\tilde e^1(t)=2\sqrt2\cos(t)(e^2+e^4),\, \tilde e^2(t)=2\sqrt2\cos(t)(e^3+ e^5),\,\tilde e^3(t)=4dt,\\
\tilde e^4(t)=-\tfrac{4\sqrt3}{3}\sin(2 t)e^8,\quad \tilde e^5(t)=4e^6,\quad\tilde e^6(t)=-4e^7,\\
\tilde e^7(t)=2\sqrt2\sin(t)(-e^2+e^4),\,\tilde e^8(t)=2\sqrt2\sin(t)(e^3-e^5).
\end{gathered}
\end{equation*}
\end{lemma}

Using these observations, we have the following result.

\begin{proposition}
The Wolf space \( \Gr_2(\bC^4) \) can be viewed as a cohomogeneity one manifold obtained by gluing together the disc bundles over two copies of the singular orbit \( \pC(2)=\SU(3)\!/\!\Lie{U}(2) \). Each principal orbit is a copy of the exceptional Aloff-Wallach space \( N^{1,0} \cong N^{1,-1} \).
The form of the metric in the two bundles is the same up to the identification
\begin{equation*}
e_1\!\mapsto-e_1,\,e_2\!\mapsto e_4,\,e_3\!\mapsto-e_5,\,e_4\!\mapsto-e_2,\,e_5\!\mapsto e_3,\,
e_6\!\mapsto e_6,\,e_7\!\mapsto-e_7,\,e_8\!\mapsto e_8.
\end{equation*}
\end{proposition}

\begin{proof}
The arguments follow those of \cref{prop:cohom1-pH2}. For convenience, let us write
down the Lie algebras of the stabilisers. For the principal orbits, we have \( \lu(1) \),
corresponding to \( e_1 \). At \( t=0 \) we see that the algebra annihilated by 
the pulled back coframe is the copy of \( \lu(2) \) spanned by \( e_1, e_8,e_2-e_4,e_3-e_5 \).
Finally, at \( t=\pi/2 \) the Lie algebra of the singular stabiliser is spanned
by the \( \lu(2) \) determined by \( e_1, e_8,e_2+e_4,e_3+e_5 \).

The identification given is the inner automorphism of \( \su(3) \) obtained by conjugating 
with the matrix
\( \left(\begin{smallmatrix}0 & i & 0 \\i & 0 & 0\\ 0 & 0 & 1\end{smallmatrix}\right) \). It
preserves the principal \( \lu(1) \) and the quaternion-k\"ahler metric. The second statement follows.
\end{proof}

As for \( \pH(2) \), note that the complex Grassmannian comes with a \( \LU(1) \)-action, now generated by the diagonal \( \LU(1) \) in \( \LU(3)\subset\SU(4) \), commuting with the \( \SU(3) \)-action. Again this circle action clearly generates a Killing vector field which
in our cohomogeneity one framework can be identified with \( e_8\in\su(3) \). 
In this case, \( X=A_1+2A_2+3A_3 \) is the sum of \( -2\sqrt3e_8 \) with \( 3A_1+3A_3 \), which commutes with \( Z=C_5 \) and is contained in \( \LK \).

\subsection{The exceptional Wolf space}
\label{sec:SU3-exc-Wolf-space}

The quaternion-k\"ahler structure on the space \( \G_2\!/\!\SO(4) \) admits a cohomogeneity one description that comes from the embedding of \( \SU(3) \) in \( \G_2 \) as the group generated by the long roots,
cf.\ \cite{Kobak-S:wolfspace}. At the Lie algebra level, we have that \( \su(3) \) is generated by 
\begin{equation*}
\begin{gathered}
A_1+A_2,\tfrac1{\sqrt2}(W_2+W_6),\tfrac1{\sqrt2}(Z_6-Z_2),\tfrac1{\sqrt2}(W_2-W_6),\\
-\tfrac1{\sqrt2}(Z_2+Z_6),W_5,-Z_5,-\tfrac1{\sqrt3}(A_1+3A_2);
\end{gathered}
\end{equation*}
this choice of basis is consistent with \eqref{eq:su3}. In order to study the orbits of the \( \SU(3) \)-action, we choose an element
\begin{equation*}
 Z\in\mathfrak{p}\cap \su(3)^\perp=\Span{W_3,W_4,Z_3,Z_4}. 
\end{equation*}
Specifically, we fix \( Z=W_3 \).
Then computations, completely similar to those in the proof of \cref{lem:pull-back-adap-frame-HP2},
give:

\begin{lemma}
Under the mapping \eqref{eq:orbmap}, Wolf's coframe pulls back to the dual of \( \su(3)\oplus\bR \) to give
\begin{equation*}
\begin{gathered}
\tilde e^1(t)=\tfrac{\sqrt2}{2}(\cos(t)^{3}-\sin(t)^{3}) e^2+\tfrac{\sqrt{2}}{2}(\cos(t)^{3}+\sin(t)^{3})e^4,\\
\tilde e^2(t)-\tfrac{\sqrt2}{2}(\cos(t)^{3}-\sin(t)^{3})e^3-\tfrac{\sqrt2}{2}(\cos(t)^{3}+\sin(t)^{3}) e^5,\\
\tilde e^3(t)=- e^6,\,\tilde e^4(t)=e^7,\,\tilde e^5(t)=\sqrt{3}dt,\, \tilde e^6(t)=-\sin(2t)e^8,\\
\tilde e^7(t)=-\sqrt{\tfrac{3}{8}}\sin(2t)(\sin(t)-\cos(t)) e^2- \sqrt{\tfrac{3}{8}}\sin(2t) (\sin(t)+\cos(t))e^4,\\
\tilde e^8(t)=-\sqrt{\tfrac{3}{8}}\sin(2t)(\sin(t)-\cos(t))e^3- \sqrt{\tfrac{3}{8}}\sin(2t)(\sin(t)+\cos(t))e^5.
\end{gathered}
\end{equation*}
\end{lemma}

With the above observations, we have the following result that
confirms statements from \cite{Gambioli:SU3action}:

\begin{proposition}
\label{prop:cohom1-excep-Wolf}
The exceptional Wolf space \( \G_2\!/\!\SO(4) \) can be viewed as a cohomogeneity one manifold obtained by gluing together disc bundles over the singular orbits \( \pC(2)=\SU(3)\!/\!\LU(2) \) and \( L=\SU(3)\!/\!\SO(3) \). Each principal orbit is an exceptional Aloff-Wallach space \( N^{1,0} \cong N^{1,-1}\). 
\end{proposition} 

\begin{proof}
Again the arguments are like those of \cref{prop:cohom1-pH2}, but for convenience we spell out
the Lie algebras of the stabilisers. For the principal orbits, we have \( \lu(1) \),
corresponding to \( e_1 \). At \( t=0 \) we see that the algebra annihilated by 
the pulled back coframe is the copy \( \lu(2) \) spanned by \( e_1, e_8,e_2-e_4,e_3-e_5 \).
Finally, at \( t=\pi/4 \) the Lie algebra of the singular stabiliser is spanned
by the copy \( \so(3) \) determined by \(e_1,e_2,e_3 \).
\end{proof}

In contrast with the quaternionic projective plane and the complex Grassmannian,
\( \G_2\!/\!\SO(4) \) clearly does not admit a (global) circle action from a commuting \( \LU(1)\subset\G_2 \)
(\( \SU(3)\subset\G_2 \) is a maximal connected subgroup). 
However, the open set corresponding to
the principal orbits does come with a circle action corresponding to \( e_8 \). Whilst \( X \) is not
a Killing vector field in this case, it turns out to satisfy the generalised condition
\begin{equation}
\label{eq:gen-Killing}
d(\norm{X}^2)\wedge \Ld_X\Omega=0,
\end{equation} 
as follows by direct computation.

We conclude our cohomogeneity one description of the Wolf spaces with an observation that in a sense ties together all three cases.

\begin{proposition}
\label{prop:equiv-CP2-end}
The vector bundle \( \SU(3)\times_{\LU(2)}\bC^2 \) over \( \pC(2) \) admits
three distinct \( \SU(3) \)-invariant quaternion-k\"ahler structures.
\end{proposition}

\begin{proof}
First note that our analysis of the cohomogeneity one \( \SU(3)
\)-actions shows that the spaces \( \pH(2)\setminus S^5 \), \(
\rm{Gr}_2(\bC^4)\setminus\pC(2) \) (one can choose either copy of \(
\pC(2) \)) and \( \G_2\!/\!\SO(4)\setminus L \) are all equivalent to
the same vector bundle \( \SU(3)\times_{\LU(2)}\bC^2 \).  The claim
then follows since the three quaternion-k\"ahler structures on \(
\pH(2) \), \( \Gr_2(\bC^4) \) and \( \G_2\!/\!\SO(4) \) induce
different structures on \( \SU(3)\times_{\LU(2)}\bC^2 \), since these
Wolf spaces have different holonomy groups and consequently different
curvature.
\end{proof}

\section{Nilpotent perturbations}
\label{sec:nil-pert}

Let \( \alpha\) be an element of \(\Lambda^p(\bR^n)^* \), and consider the \emph{(affine) perturbation} by a fixed
\( p \)-form \( \delta \), meaning
\begin{equation*}
\beta(t)=\alpha+t\delta,\quad t\in\bR.
\end{equation*}
Generally, it is hard to decide whether \( \beta(t) \) and \( \alpha \) lie
in the same \( \GL(n,\bR) \)-orbit for all \( t \). However, a 
useful sufficient criterion can be phrased as follows.

\begin{proposition}
\label{prop:lin-deform}
Let \( A\in\gl(n,\bR) \). If the associated derivation \( \rho(A) \) satisfies \( \rho(A)^2\alpha=0 \), then 
\begin{equation*}
\beta(t)=\alpha+t\rho(A)\alpha
\end{equation*}
lies in the same \( \GL(n,\bR) \)-orbit as \( \alpha \) for all \( t\in\bR \).
\end{proposition}

\begin{proof}
The proof is elementary. We expand \( g(t)=\exp(tA) \) to find that
\begin{equation*}
g(t)\alpha=\alpha+t\rho(A)\alpha,
\end{equation*}
since the higher order terms \( \tfrac{t^k}{k!}\rho(A)^k\alpha \), \( k\geqslant2 \), vanish by assumption.
So \( \beta(t) \) and \( \alpha \) lie in the same \( \GL(n,\bR) \)-orbit, as claimed. 
\end{proof}

Motivated by \cref{prop:lin-deform}, we would like to characterise the perturbations of \( \alpha \)
that are parametrised by solutions of
\begin{equation}
\label{eqn:rhosquarealpha}
\rho(A)^2\alpha=0.
\end{equation}
Amongst these solutions we obviously have elements of the stabiliser \( \g \) of \( \alpha \) in \( \gl(n,\bR) \),
but these give rise to trivial perturbations \( \beta(t)\equiv\alpha \). In order to eliminate this indeterminacy we observe the following:

\begin{proposition}
\label{prop:nilp-charac}
Let \( \alpha\in\Lambda^p(\bR^n)^* \). Then every solution \( A \) of \eqref{eqn:rhosquarealpha} satisfies
\begin{equation*}
\rho(A)\alpha=\rho(N)\alpha,
\end{equation*}
where \( N \) is a nilpotent solution of \eqref{eqn:rhosquarealpha}.
\end{proposition}

\begin{proof}
Over the complex numbers we can put \( A \), as an endomorphism of \( (\bC^n)^* \), into Jordan form. Correspondingly, we obtain a direct sum decomposition 
\( (\bC^n)^*=\bigoplus V_{i} \), where \( V_i \) is the generalised eigenspace relative to the eigenvalue \( \lambda_i \).
Denoting by \( I_{V_i} \) the matrix corresponding to the projection onto \( V_i \), we have 
\( A= N + \sum \lambda_i I_{V_i} \), where \( N \) is nilpotent and real.

In accordance with the above, we can also decompose \( \Lambda^p(\bC^n)^* \) as a direct sum
\begin{equation}
\label{eqn:decomposition}
\bigoplus_{k_1+\dots + k_m=p}\>\bigoplus_{j_1<\dotsb< j_m} \Lambda^{k_1}V_{j_1}\otimes \dotsb \otimes \Lambda^{k_m}V_{j_m},
\end{equation}
where each summand is closed under \( \rho(A) \). Moreover, direct computation shows that
\begin{equation*}
\begin{gathered}
(\rho(A)-(\lambda_1+\lambda_2)I)(\alpha_1\wedge \alpha_2)\\
=(\rho(A)-\lambda_1I)\alpha_1\wedge \alpha_2 
+\alpha_1\wedge (\rho(A)-\lambda_2I)\alpha_2,
\end{gathered}
\end{equation*}
giving that each summand in \eqref{eqn:decomposition} is contained in the generalised eigenspace of \( k_1\lambda_{j_1}+\dots + k_m\lambda_{j_m} \) relative to \( \rho(A) \). It follows that \( \rho(A)^2 \) also preserves the decomposition \eqref{eqn:decomposition}, and its kernel is contained in 
\begin{equation*} 
\bigoplus_{j_1<\dotsb< j_m}\bigoplus_{\stackrel{k_1+\dots + k_m=p}{k_1\lambda_{j_1}+\dotsb+k_m\lambda_{j_m}=0}} \Lambda^{k_1}V_{j_1}\otimes \dotsb \otimes \Lambda^{k_m}V_{j_m}.
\end{equation*}
On this space, \( \rho(A) \) and \( \rho(N) \) act in the same way. Therefore, if \( A \) is a solution of \cref{eqn:rhosquarealpha}, then 
\begin{equation*}
\rho(N)^2\alpha=0 \quad \textrm{and}\quad\rho(N)\alpha=\rho(A)\alpha.\qedhere
\end{equation*}
\end{proof}

Motivated by \cref{prop:nilp-charac}, we shall restrict our attention to nilpotent solutions of \eqref{eqn:rhosquarealpha}
and refer to these as \emph{nilpotent perturbations}.

Whilst \cref{prop:nilp-charac} is valid in general,
our interest is the case where \( \alpha \) is the quaternionic form \eqref{eq:quat_4_form}. 
If \( e^0_1,\ldots,e^0_8 \) denotes a standard basis of \( \bR^8 \), and \( e^i_0 \) its dual, it turns out 
that solutions to the perturbation problem are most conveniently expressed in terms of the orthonormal basis \( (E^1,\dots, E^8) \) 
of \( (\bR^8)^* \) equal to
\begin{equation}
\label{eq:adap-pert-basis}
(e^8_0,\tfrac{\sqrt{3}}{2}e^2_0-\tfrac{1}{2} e^6_0,-\tfrac{\sqrt{3}}{2}e^1_0-\tfrac{1}{2} e^5_0,
-\tfrac{\sqrt{3}}{2}e^5_0+\tfrac{1}{2} e^1_0,\tfrac{\sqrt{3}}{2}e^6_0+\tfrac{1}{2} e^2_0,- e^4_0,e^3_0,e^7_0).
\end{equation}
In these terms, \eqref{eq:quat_4_form} reads
\begin{equation*}
\begin{split}
\Omega = -E^{1247}+\sqrt{3} E^{1248}-E^{1256}-E^{1346}&+E^{1357}+\sqrt{3} E^{1358}\\
+2E^{1458}-E^{1678}-E^{2345}+2E^{2367}&-\sqrt{3} E^{2467}+E^{2468}-E^{2578}\\
&-E^{3478}-\sqrt{3} E^{3567}-E^{3568}.
\end{split}
\end{equation*}

A significant observation, that may at first not be fully appreciated,
is that the stabiliser of \( E^{123} \) in \( \Sp(2)\Sp(1) \) is \(
\SO(3) \). Computations give:

\begin{lemma}
\label{lem:stab-3plane}
The stabiliser of \( \Span{E^1,E^2,E^3} \) in \( \Sp(2)\Sp(1) \) is 
the copy of \( \SO(3) \) whose Lie algebra is spanned by  the elements
\begin{equation*}
\left(\begin{smallmatrix} 
0 & \frac{\sqrt3}2j \\ 
\frac{\sqrt3}2j & -j 
\end{smallmatrix}\right)-\tfrac12R_j, \quad
\left(\begin{smallmatrix} 0 & \frac{\sqrt3}2k \\ 
\frac{\sqrt3}2k & k 
\end{smallmatrix}\right)-\tfrac12R_k,\quad 
\left(\begin{smallmatrix} 
\frac32i & 0 \\ 
0 &- \frac12i 
\end{smallmatrix}\right)-\tfrac12R_i
\end{equation*}
of \( \lsp(2)\oplus\lsp(1) \).
\end{lemma}

\noindent Above, \( R_i \) denotes right multiplication by \( i \), and so forth.

Prompted by \cref{lem:stab-3plane}, we shall decompose \( (\bR^8)^* \) 
as the sum of two irreducible \( \SO(3) \)-modules
\begin{equation*}
S^2=\Span{E^1,E^2,E^3}, \quad S^4=\Span{E^4,\dots, E^8}.
\end{equation*}
Using the dual basis \( E_i \), we then have that the \( 2 \)-forms
\begin{equation}
\label{eqn:betafromomega}
\beta^1:=(E_2\wedge E_3)\hook\Omega,\quad 
\beta^2:=(E_3\wedge E_1)\hook\Omega,\quad
\beta^3:=(E_1\wedge E_2)\hook\Omega
\end{equation}
define an \( \SO(3) \)-equivariant linear map \( S^2\to \Lambda^2S^4 \).

In these terms, the following result describes nilpotent
perturbations in the quaternionic setting.

\begin{theorem}
\label{thm:nilpotent-pert}
Up to the action of \( \Sp(2)\Sp(1) \), nilpotent solutions of 
\begin{equation*}
\rho(A)^2\Omega=0
\end{equation*}
are parametrised by linear maps \( v\colon S^2\to S^4 \) such that 
\begin{equation*}
v^2\wedge v^3\wedge \beta^1+v^3\wedge v^1\wedge \beta^2+v^1\wedge v^2\wedge \beta^3=0.
\end{equation*}
Explicitly \( v \) corresponds to the endomorphism
\( A=\sum_{i=1}^3\!v^i\otimes E_i \) of \( (\bR^8)^* \).
\end{theorem}

\noindent In terms of forms, the notation \( v^i\otimes E_i \) above represents
the endomorphism 
\begin{equation*}
\gamma\mapsto v^i\wedge(E_i\hook\gamma).
\end{equation*}

\begin{remark}
  As a corollary of \cref{thm:nilpotent-pert}, nilpotent
  solutions of \( \rho(A)^2\Omega=0 \) actually satisfy \( A^2=0
  \). Note that \( \rho \) is not an algebra homomorphism, so \( A^2=0
  \) does not imply \( \rho(A)^2=0 \).
\end{remark}

The \( 3 \)-dimensional subspace \( \Span{E_1,E_2,E_3} \) of \( \bR^8 \), i.e.\ the
annihilator of \( S^4 \), is uniquely determined up to the \(
\Sp(2)\Sp(1) \)-action. Subspaces of \( \bR^8 \) in its \(
\Sp(2)\Sp(1) \)-orbit can be characterised by the angle between
quaternionic lines. Indeed, consider the \( \Sp(2)\Sp(1) \)-invariant
function
\begin{equation*}
Q\colon \pH^1\times \pH^1\to\bR, \quad Q([v],[w]) = \max_{J\in \Sp(1)} \frac{\langle v,Jw\rangle^2}{\abs{v}^2\abs{w}^2}.
\end{equation*}
In an affine chart we can express \( Q \) as
\begin{equation*}
Q([1:p],[1:q]) = \frac{\abs{1+\overline{p}q}^2}{\abs{1+\overline{p}q}^2 + \abs{q-p}^2}.
\end{equation*}

\begin{lemma}
\label{lemma:quaternionicangle}
Let \( v,w\in \bR^8 \) be two orthogonal non-zero vectors such that
\begin{equation*}
((v\wedge w)\hook \Omega)^3=0.
\end{equation*}
Then \( Q([v],[w])=\tfrac14 \).
\end{lemma}

\begin{proof}
Using the action of the quaternionic unitary group, we can assume that 
\( w=ae_8^0 \) and \( v=be_1^0+ce_5^0 \). Then a straightforward computation shows that 
\( b=\pm\sqrt3 c \), giving the asserted result.
\end{proof}

\begin{lemma}
\label{lemma:orthonormalvi}
Let \( w_1,\ldots,w_k \) be orthonormal vectors in \( \bR^8 \) such that 
\begin{equation*}
((w_i\wedge w_j)\hook\Omega)^3=0.
\end{equation*}
Then \( k\leqslant 3\), and if \( k=3 \), up to the action of \( \Sp(2)\Sp(1) \), we may assume that
\( w_i=E_i \) for \( i=1,2,3 \).
\end{lemma}

\begin{proof}
By \cref{lemma:quaternionicangle}, we have that
\begin{equation}
 \label{eqn:equalquaternionicangles}
Q([w_i],[w_j])=\tfrac14, \quad i\neq j.
\end{equation}
If \( k>2 \), the points \( [w_i] \) in \( \pH^1 \) satisfy \eqref{eqn:equalquaternionicangles}, and we can assume, up to \( \Sp(2)\Sp(1) \)-action, that 
\begin{equation*}
[w_1]=[1:0], [w_2]=[1:p], [w_3]=[1:q].
\end{equation*}
Then 
\begin{equation*}
\abs{p}^2=3=\abs{q}^2, \quad \abs{p-q}^2=3\abs{1+\overline{p}q}^2
\end{equation*}
which only has the solution \( p=-q \). It is clearly not possible to add a fourth element \( [w_4] \) 
so that \eqref{eqn:equalquaternionicangles} is satisfied.

Assuming then that \( k=3 \), we can use the action of \( \Sp(2)\Sp(1) \), as in the proof of \cref{lemma:quaternionicangle}, to obtain \( w_1=e_8^0 \). 
This leaves us with an \( \SO(4) \) symmetry that can be used to obtain \( w_2\in\Span{e_1^0,e_2^0,e_3^0,e_4^0,e_6^0} \). 
The stabiliser in \( \SO(4) \) of \( e_6^0 \) is \( \LU(2) \) and up to this \( \LU(2) \)-action, we can assume \( w_2=E_2 \). 
The condition \( p=-q \) together with orthogonality implies that \( w_3 \) is in the span of 
\( \sqrt{3}e_1^0+e_5^0 \) and \( \sqrt{3}e_3^0+e_7^0 \).
The stabiliser of \( E_2 \) in \( \LU(2) \), isomorphic to \( \LU(1) \), acts non-trivially on this \( 2 \)-dimensional space which allows us to set \( w_3=E_3 \).
\end{proof}

\begin{proof}[Proof of \cref{thm:nilpotent-pert}]
Up to change of basis nilpotent matrices are classified over the reals by partitions with weight \( 8 \), giving \( 22 \) possibilities that
can be encoded in terms of Young diagrams.
For example, the diagram 
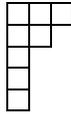
\begin{equation*}
\tiny\yng(3,2,1,1,1)
\end{equation*}
describes the endomorphisms of \( (\bR^8)^* \) with Jordan blocks of size \( 3,2,1,1,1 \),
that with respect to some basis \( \{w^1,\dotsc, w^8\} \) satisfy
\( w^3\mapsto w^2,\ w^2\mapsto w^1 \), \( w^5\mapsto w^4 \), with the other
vectors mapped to zero. For each diagram \( \Gamma \) we can fix
a representative endomorphism \( A_\Gamma \) and compute the space
\begin{equation*}
K_{\Gamma}=\left\{\alpha\in \Lambda^4(\bR^8)^*\colon\, \rho(A_\Gamma)^2\alpha=0\right\}.
\end{equation*}
The equation \( \rho(A)^2\Omega=0 \) has a solution with diagram \( \Gamma \) if \( \rho(A_\Gamma)^2\alpha=0 \) for some \( \alpha \) in the orbit of \( \Omega \),
and this requires that for each nonzero \( v\wedge w \) in \( \Lambda^2\bR^8 \) the map
\begin{equation*}
K_\Gamma \to \Lambda^4(\bR^8)^*, \quad \alpha \mapsto\left((v\wedge w)\hook \alpha\right)^2
\end{equation*}
is not identically zero. Computations show that this rules out all cases except
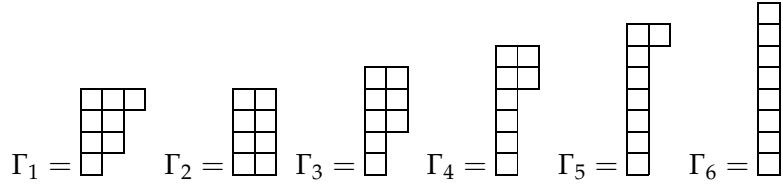
\begin{equation*}
\tiny
\Gamma_1=\yng(3,2,2,1) \quad
\Gamma_2=\yng(2,2,2,2) \quad
\Gamma_3=\yng(2,2,2,1,1) \quad
\Gamma_4=\yng(2,2,1,1,1,1)\quad
\Gamma_5=\yng(2,1,1,1,1,1,1) \quad
\Gamma_6=\yng(1,1,1,1,1,1,1,1)
\end{equation*}
The last diagram corresponds to \( A=0 \) and \( \Gamma_5 \) corresponds to \( A \) being any rank one nilpotent matrix. In either case, the statement of the theorem holds.

Now let \( A \) be a solution of \( \rho(A)^2\Omega=0 \). In terms of
its diagram, let \( k \) be the number of rows of length greater than
one, and reorder the associated basis in order that \( w^i \)
corresponds to the rightmost box in the \( i \)th row for \( 1\leqslant i\leqslant k \). 
In other words, the elements \( w_1,\ldots, w_k \) of the dual
basis span the annihilator \( (\ker A+\im A)^\mathrm{o} \).  Then each
\( (w_i\wedge w_j)\hook \Omega \) is degenerate, meaning that \(
((w_i\wedge w_j)\hook \Omega)^3=0 \). In the case of \( \Gamma_4 \)
this holds because
\begin{equation*}
2(Aw^1)\wedge(Aw^2)\wedge (w_1\hook w_2\hook\Omega)=\rho(A)^2\Omega=0.
\end{equation*}
For \( \Gamma_1,\Gamma_2, \Gamma_3 \) the hypothesis implies degeneracy as it forces \( w_i\hook w_j\hook w_\ell\hook \Omega\) to be zero.

Without loss of generality, we can assume that the covectors \( w^i \) are orthogonal to \( \im A \) and orthonormal. From \cref{lemma:orthonormalvi} we conclude that \( k\leqslant 3 \), so that we can rule out \( \Gamma_2 \), and assume \( w_i=E_i \). In the case of \( \Gamma_3 \) and \( \Gamma_4 \), it now suffices to write
\( A=\sum_{i=1}^3\!v^i\otimes E_i \),
so that, using \eqref{eqn:betafromomega}, the vanishing of \( \rho(A)^2\Omega \)
becomes
\begin{equation*}
v^{23}\wedge \beta^1+v^{31}\wedge \beta^2+v^{12}\wedge \beta^3=0,
\end{equation*}
as required.

Finally, in order to rule out the case \( \Gamma_1 \), assume the associated basis has the form \( E_1,E_2,E_3,v_1,\dotsc, v_5 \), so that with obvious notation 
\begin{equation*}
A = v^1\otimes E_1+v^2\otimes E_2 + v^3\otimes E_3+v^4\otimes v_1.
\end{equation*} 
Then
\begin{equation}
\begin{split}
 \label{eqn:case3221}
0&=\rho(A)^2\Omega \\
&=2\!\sum_{1\leqslant i<j\leqslant 3} v^i\wedge v^j \wedge(E_i\hook E_j\hook \Omega) + 
v^4\wedge (E_1\hook\Omega - 2\sum_{i=1}^3 v^i\wedge (E_i\hook v_1\hook\Omega)).
\end{split}
\end{equation}
Wedging with \( v^4 \) and using \eqref{eqn:betafromomega}, we get
\begin{equation}
 \label{eqn:cyclicsumzero}
v^{124}\wedge\beta^3+v^{314}\wedge\beta^2+v^{234}\wedge\beta^1=0.
\end{equation}
Projecting \cref{eqn:case3221} onto the space \( S^2\otimes \Lambda^3S^4 \), we find
the condition
\begin{equation*}
\begin{gathered}
v^4\wedge \Big[E^2\wedge\beta^3-E^3\wedge\beta^2 
 +  2v^1\wedge (v_1\hook (E^2\wedge\beta^3-E^3\wedge\beta^2))\phantom{mmm}\\[-5pt]
+  2v^2\wedge (v_1\hook (E^3\wedge\beta^1-E^1\wedge\beta^3))
+  2v^3\wedge (v_1\hook (E^1\wedge\beta^2-E^2\wedge\beta^1))\Big] = 0.
\end{gathered}
\end{equation*}
This implies that
\begin{equation}
\label{eqn:v4wedgebetas}
\begin{gathered}
 0=v^4\wedge\big[\beta^2 + 2v^1\wedge (v_1\hook \beta^2)-2v^2\wedge (v_1\hook \beta^1)\big]
\quad\textrm{and}\\
 0=v^4\wedge\big[\beta^3+2v^1\wedge (v_1\hook \beta^3)-2v^3\wedge (v_1\hook \beta^1)\big].
\end{gathered}
\end{equation}
Therefore \( v^{124}\wedge \beta^3=v^{314}\wedge \beta^2=2v^{234}\wedge\beta^1 \),
and \eqref{eqn:cyclicsumzero} implies each term is zero. Then \cref{eqn:v4wedgebetas} gives
\(
v^4\wedge \beta^2, v^4\wedge \beta^3\in \Span{v^{234}}
\),
which is absurd. In conclusion, \( \Gamma_1 \) cannot occur, 
and the proof is complete.
\end{proof}

\subsection{\( \LU(1) \)-invariant perturbations}

If we impose invariance, \cref{thm:nilpotent-pert} can be simplified considerably.
Indeed, consider the \( 8 \)-dimensional representation of \( \LU(1) \) that
models the tangent space to the open set formed of our principal orbits (cf.\ \cref{eq:tspace-prc-orb}).
The nilpotent perturbations compatible with this action have a simple description.

\begin{proposition}
\label{prop:U1lindef}
Let \( \Omega\in\Lambda^4(\bR^8)^* \) be a \( \LU(1) \)-invariant \( 4 \)-form with stabiliser group \( \Sp(2)\Sp(1) \). 
Then there is an orthonormal basis as in \eqref{eq:adap-pert-basis}, such that (for the dual basis) 
\begin{equation*}
E_1,E_8\in 2\bR, \quad 2 V_1=\Span{E_2,E_3}\oplus\Span{E_4,E_5}, \quad V_2=\Span{E_6,E_7}.
\end{equation*}
The space of \( \LU(1) \)-invariant nilpotent perturbations is generated by
\begin{equation}
\label{eq:U1-inv-pert-gen}
E^8\wedge (E_1\hook\Omega).
\end{equation}
\end{proposition}

\begin{proof}
By construction the \( 4 \)-form is fixed by \( \LU(1) \) so that this group lies inside \( \Sp(2)\Sp(1) \). 
We can then find an adapted basis as in \eqref{eq:adap-pert-basis}, 
such that \( \LU(1)\subset\SO(3)\subset\Sp(2)\Sp(1) \), where the middle subgroup \( \SO(3) \) 
preserves the splitting
\begin{equation*}
\Span{E^1,E^2,E^3}\oplus\Span{E^4,\dots, E^8}=S^2\oplus S^4;
\end{equation*}
this is because all subgroups \( \LU(1) \) in \( \Sp(2)\Sp(1) \) are conjugate, so we can 
assume that \( \LU(1) \) is contained in some conjugate of \( \SO(3) \). 
Due to the way this \( \LU(1) \) acts on \( S^4 \), we deduce that for the dual basis one has
\begin{equation*}
\Span{E_1,E_2,E_3}=\bR\oplus V_1, \quad \Span{E_4,\dots, E_8}=\bR\oplus V_1\oplus V_2.
\end{equation*}
Now, by making a change of basis if necessary, we can assume that 
\begin{equation*}
E_1\in\bR, \quad V_1=\Span{E_2,E_3},
\end{equation*}
corresponding to \( \LU(1) \) stabilising \( E^1 \) in \( \SO(3) \), whose precise form can be recovered from Lemma~\ref{lem:stab-3plane}. 
Computing its action on \( S^4 \), we find
\begin{equation*}
E_8\in\bR, \quad V_1=\Span{E_4,E_5}, \quad V_2=\Span{E_6,E_7}.
\end{equation*}

Since on nilpotent matrices the map \( A\mapsto \rho(A)\Omega \) is injective,
the latter is invariant if and only if \( A \) is invariant. This means that the space of invariant nilpotent perturbations is given by 
\begin{equation*}
v^1\in \Span{E^8},\quad v^2,v^3\in \Span{E^4,E^5},
\end{equation*}
where
\begin{equation*}
\beta^1\wedge v^{23}+\beta^2\wedge v^{31}+\beta^3\wedge v^{12}=0.
\end{equation*}
It follows that \( v^2\wedge v^3=0 \), and by invariance this means that \( v^2=0=v^3 \).
In conclusion, the space of \( \LU(1) \)-invariant perturbations is generated by
\eqref{eq:U1-inv-pert-gen}, as required.
\end{proof}

\section{New closed \( \Sp(2)\Sp(1) \)-structures}
\label{sec:main-result}

We are now ready to produce explicit examples of closed \(
\Sp(2)\Sp(1) \)-structures. Since the corresponding exterior
differential system is effectively underdetermined, it is not
surprising that, at least locally, it is possible to obtain such
examples by deforming the quaternion-k\"ahler metric on a Wolf space
\( M \). In fact, it follows from results of the first author
\cite{Conti:EmbeddingIntoManifolds} that if one considers the induced
structure on a real analytic hypersurface \( N\subset M \) (in the
language of \cite{Conti-M:Harmonic}, an \( \SO(4) \)-structure with a
closed \( 4 \)-form \( \beta \)), then one can extend it to obtain a
closed \( \Sp(2)\Sp(1) \)-structure in a neighbourhood of \( N \).

It is not difficult to see that there is more flexibility than that
arising from local diffeomorphisms. In our cohomogeneity one setting,
this indeterminacy can be seen by parametrising invariant forms in the
\( \GL(8,\bR) \)-orbit of the quaternion-k\"ahler \( 4 \)-form. These
depend on \( 11 \) functions, because relatively to
\eqref{eq:tspace-prc-orb}, the centralizer of \( \LU(1) \) in \(
\GL(8,\bR) \) has dimension \( 14 \) and intersects \( \Sp(2)\Sp(1) \)
in a 3-dimensional torus. Explicit computations show that
closedness of the form corresponds to \( 7 \) equations, leaving \( 4
\) undetermined functions, whilst equivariant diffeomorphisms only
depend on one function.

Whilst the discussion above emphasises local flexibility, the method of Section~\ref{sec:nil-pert} proves to be a particularly
useful approach to obtain examples that are both explicit and global. As it
turns out, each Wolf space has a family of closed nilpotent perturbations determined by the vector field \( X \), corresponding to
\( e_8 \). The perturbed metric happens to be genuinely different from the original only when this vector field is not Killing.
 
\begin{lemma}
\label{lemma:deformation_vector_field}
On each of the three Wolf spaces, \( \SU(3) \)-invariant closed nilpotent perturbations of the quaternion-k\"ahler structure \( \Omega_{qK} \) have the form
\begin{equation*}
\widetilde\Omega=\Omega_{qK}+dh\wedge (e_8\hook\Omega),
\end{equation*}
where \( h \) is any smooth \( \SU(3) \)-invariant function.
\end{lemma}

\begin{proof}
An \( \SU(3) \)-invariant perturbation is defined by an \( \SU(3) \)-invariant section of 
\( \End(T(\G\!/\!\LK)) \). On the complement of the singular orbits
\begin{equation*}
\SU(3)\!/\!\LU(1)\times (0,T),
\end{equation*}
\cref{prop:U1lindef} implies that the perturbation must be induced by a \( t \)-dependent nilpotent endomorphism of \(\Span{dt,e^8} \). 
Concretely, the perturbation must be of the form
\begin{equation*}
\Omega_{qK}\pm(\lambda e^8 + \mu dt)\wedge \left((\mu e_8 - \lambda \D{}{t})\hook\Omega\right).
\end{equation*}

Insisting that the perturbed \( 4 \)-form is closed forces \( \lambda \) to vanish, since for all three quaternion-k\"ahler metrics the restriction of \( e^8\wedge(\D{}{t}\hook\Omega) \) to principal orbits is not closed. On the other hand perturbations of the form \( f(t)dt\otimes e_8 \) preserve closedness, because \( \Ld_{e_8}\Omega\wedge dt \) is zero in each case.

Having resolved the problem on the complement of singular orbits, we need to address the conditions that ensure that our solution will extend.
Recall that there are three basic models to consider, summarised by \cref{tab:models}. 
Let us first consider the vector bundle \( \SU(3)\times_{\LU(2)} \bC^2 \), where \( \bC^2 \) is the standard representation of \( \LU(2) \), corresponding to a tubular neighbourhood of each singular orbit \( \pC(2) \). Away from the zero section, \( e_8 \) defines an invariant vector field. Since \( e_8\in\lu(2) \), this vector field is vertical, i.e. it is a \( \LU(2) \)-invariant vector field on \( \bC^2 \). In appropriate real coordinates \( (x,y,z,w) \), we can write
\begin{equation*}
e_8=\tfrac{2\sqrt3}{3}\Big(x\D{}{y}-y\D{}{x}+z\D{}{w}-w\D{}{z}\Big),\quad tdt= xdx+ydy+zdz+wdw,
\end{equation*}
where the factor in front is due to the period of this \( \LU(2) \)-invariant vector field.

This shows that \( f(t)dt\otimes e_8 \) extends smoothly if and only if \( t\mapsto f(t)/t \) is smooth and even or, equivalently, \( t\mapsto f(t) \) is smooth and odd. At tubular neighbourhoods of the other possible singular orbits, corresponding to the vector bundles \( \SU(3)\times_{\SU(2)}\Sigma^2 \) and \( \SU(3)\times_{\SO(3)}\bR^3 \), respectively, \( e_8 \) defines an invariant direction in the isotropy representation. This means that \( f(t)dt\otimes e_8 \) is smooth if so is \( f(t)dt \). Consequently, \( t\mapsto f(t) \) must again be a smooth odd function.
In summary, \( f \) extends to a smooth function on \( \bR \) that satisfies
\begin{equation*}
f(t)=-f(-t), \quad f(T-t)=-f(T+t).
\end{equation*}
Any primitive \( h \) of \( f(t)dt \) then satisfies
\( h(t)=h(-t) \) and \( h(T-t)=h(T+t) \) and therefore defines a global \( \SU(3) \)-invariant function on \( \G\!/\!\LK \),
as required.
\end{proof}

\subsection{Perturbing with a Killing vector field}

When \( X \) is Killing, we have rigidity in the sense
that nilpotent perturbation just results in different, but \( \SU(3) \)-equivalent,
ways of expressing Wolf's quaternion-k\"ahler structure:

\begin{proposition}
Applying \( \SU(3) \)-invariant nilpotent perturbations to the quat\-ernion-k\"ahler structure on \( \pH(2) \) 
and \( \Gr_2(\bC^4) \) leave the structures unchanged up to \( \SU(3) \)-equivariant isometry.
\end{proposition}

\begin{proof}
It follows by Lemma~\ref{lemma:deformation_vector_field} that each nilpotent perturbation can be associated with an \( \SU(3) \)-invariant function \( h \) and consequently an invariant vector field \(-h(t)e_8\). Away from the singular orbits, its flow has the form
\begin{equation*}
\phi^s\colon \SU(3)\times(0,T)\to \SU(3)\times(0,T), \quad (g,t)\mapsto (g\exp (-sh(t)e_8),t).
\end{equation*}
Then \( \phi^1 \) is an equivariant diffeomorphism whose differential at \( (e,t) \) is given by
\begin{equation*}
(v,0)\mapsto (\Ad(\exp(h(t)e_8))v,0), \quad \frac{\partial}{\partial t}\mapsto -h'(t)e_8+\frac{\partial}{\partial t}.
\end{equation*}
Since the adjoint action of \( e_8 \) preserves the quaternion-k\"ahler metric, we obtain the same metric up to the isometry that
corresponds to replacing \( e^8 \) by \( e^8+h'(t)dt \), as required.
\end{proof}

\subsection{Perturbing the exceptional Wolf space}

As discussed \( X \) is not Killing in the case of \( \G_2\!/\!\SO(4) \), but does satisfy the condition \eqref{eq:gen-Killing}.
In a sense this is exactly what is needed to obtain non-trivial perturbation results. 

\begin{theorem}
\label{thm:mainresult}
The exceptional Wolf space \( \G_2\!/\!\SO(4) \) admits \( \SU(3) \)-invariant non-Einstein positive harmonic \( \Sp(2)\Sp(1) \)-structures. 
The \( 4 \)-form determining each such structure belongs to the same cohomology class
as the quaternion-k\"ahler \( 4 \)-form. 
\end{theorem}

\begin{proof}
By Lemma~\ref{lemma:deformation_vector_field}, any \( \SU(3) \)-invariant function \( h \) defines a closed perturbation
\begin{equation*}
\Omega + dh\wedge (e_8\hook\Omega).
\end{equation*}

In order to verify that we get non-Einstein examples, we compute the Ricci tensor which equals
{{\tiny{
\begin{equation*}
\left(\begin{smallmatrix}8-\frac{1}{3}\tan(2t)^{2} h'(t)^{2}&0&0&0&0&0&-\frac{1}{6} \frac{ \sqrt{3} h'(t)^{2} {(3+\cos(4t))}}{\cos(2t)^{2}}&-\frac{1}{3}\tan(2t) h''(t)-4 h'(t)\\0&8-\frac{1}{3}\tan(2t)^{2} h'(t)^{2}&0&0&0&0&-\frac{1}{3}\tan(2t) h''(t)-4 h'(t)&\frac{1}{6} \frac{ \sqrt{3} h'(t)^{2} {(3+\cos(4t))}}{\cos(2t)^{2}}\\0&0&8&0&0&0&0&0\\0&0&0&8&0&0&0&0\\0&0&0&0&8-\frac{4}{3}\tan(2t)^{2} h'(t)^{2}&\frac{4}{3}\tan(2t)^{2} \sqrt{3} h'(t)&0&0\\0&0&0&0&\frac{4}{3}\tan(2t)^{2} \sqrt{3} h'(t)&8&0&0\\-\frac{1}{6} \frac{ \sqrt{3} h'(t)^{2} {(3+\cos(4t))}}{\cos(2t)^{2}}&-\frac{1}{3}\tan(2t) h''(t)-4 h'(t)&0&0&0&0&8+\frac{1}{3}\tan(2t)^{2} h'(t)^{2}&0\\-\frac{1}{3}\tan(2t) h''(t)-4 h'(t)&\frac{1}{6} \frac{ \sqrt{3} h'(t)^{2} {(3+\cos(4t))}}{\cos(2t)^{2}}&0&0&0&0&0&8+\frac{1}{3}\tan(2t)^{2} h'(t)^{2}\end{smallmatrix}\right).
\end{equation*}
}}}
Finally, note that this tells us that particular the scalar curvature is
\begin{equation}
\label{eq:scalar-pertb}
s = 64-\tfrac{4}{3}\tan(2t)^{2} h'(t)^{2},
\end{equation}
so that we can get \( s>0 \), but generally non-constant, by choosing \( h \) 
suitably.

For the final statement, we notice that the cohomology class of a closed \( 4 \)-form on \( \G_2\!/\!\SO(4) \) is determined by its restriction to the singular orbit \( \pC^2 \), which is a quaternionic submanifold. Since explicit verification shows that the quaternion-k\"ahler form and perturbed \( 4 \)-form both restrict to the volume form of \( \pC(2) \), we conclude that
they belong to the same cohomology class, as required.
\end{proof}

\begin{remark}
Different choices of the perturbing function \( h \) in \cref{thm:mainresult} yield non-isometric metrics. Indeed, let \( \phi \) be an isometry between two such metrics. As both metrics have isometry group \( \SU(3) \), \( \phi \) maps \( \SU(3) \)-orbits to \( \SU(3) \)-orbits. In addition, corresponding orbits must have the same volume; notice that regardless of \( h \), the volume of a principal orbit \( \SU(3)\!/\!\LU(1)\times\{t\} \) is a constant multiple of \( \sin(2t)^3\cos(2t)^2 \). Therefore, the perturbing functions must coincide up to a constant.
\end{remark}

\begin{remark}
As \( \SU(3) \)-invariant functions on \( \G_2\!/\!\SO(4) \) are in one-to-one correspondence with smooth even functions of period \( \pi/2 \) on the reals, 
it is easy to find explicit closed perturbations where \( h \) is real-analytic.
\end{remark}
	
It is clear from \cref{eq:scalar-pertb} that on \( \G_2\!/\!\SO(4) \) we only have constant scalar curvature in the quaternion-k\"ahler case.
If we are willing to remove the singular orbit \( \pC(2) \), however, the conclusion changes:

\begin{corollary}
The vector bundle \( \SU(3)\times_{\SO(3)}\bR^3 \) admits non-Einstein harmonic \( \Sp(2)\Sp(1) \)-structures with constant scalar curvature.
\end{corollary}

\begin{proof}
It follows from the proof of \cref{thm:mainresult} that by choosing \( h(t)=c\log \sin (2t) \), for any \( c\in\bR \), we obtain an incomplete closed \( \Sp(2)\Sp(1) \)-structure
of constant scalar curvature defined on \( \G_2\!/\!\SO(4)\setminus\pC(2) \). For \( c\neq0 \), this structure is non-Einstein. 
\end{proof}

\section{Relations to other special geometries}

In our list of symmetric spaces with a cohomogeneity one \( \SU(3) \)-action one is missing, 
namely the Lie group \( \SU(3) \) itself, realised as the coset space \( \SU(3)^2\!/\!\Delta\SU(3) \). 
The relevant action is by consimilarity \cite{Horn-J:con-ac}
\begin{equation*}
\SU(3)\times\SU(3)\to\SU(3)\colon\, (g,h)\mapsto gh\bar{g}^{-1}=ghg^T.
\end{equation*}
This action preserves the parallel \( \PSU(3) \)-structure, given by
\begin{equation}
\label{eq:PSU3str}
\gamma=\tfrac16\sum_{i=1}^8e^i\wedge de^i
\end{equation}
in terms of our usual basis \( e^1,\ldots,e^8 \) of \( \su(3)^* \). 

Computations, similar to those of \cref{sec:cohom-one-ac}, reveal that there are
equi\-variant isomorphisms \( \pH(2)\setminus \pC(2)\cong \SU(3)\setminus L \), and
\( \G_2\!/\!\SO(4)\setminus \pC(2) \) \( \cong \SU(3)\setminus S^5 \). A priori, the
latter identification would seem to suggest the possibility of
using the techniques of \cref{sec:nil-pert} to find new harmonic \( \PSU(3) \)-structures, 
as studied by Hitchin \cite{Hitchin:Stableformsand}.
In fact, as for the exceptional Wolf space, \( \SU(3) \) has a ``hidden'' \( \LU(1) \)-action,
which has a natural interpretation in terms of the fibres of the equivariant map
\begin{equation*}
\SU(3)\ni P\mapsto P\overline{P},
\end{equation*}
which intertwines action by consimilarity and conjugation.

Computations show:

\begin{proposition}
There are no non-trivial \( \SU(3) \)-invariant harmonic nilpotent perturbations
of the \( \PSU(3) \)-structure \eqref{eq:PSU3str} on \( \SU(3) \). 
\end{proposition}

Our studies are also related to \( \G_2 \)-holonomy metrics. The
starting point is the quotient of the quaternionic projective plane by
the circle action generated by the Killing vector field \( X \). More
specifically, one has the \( \SU(3) \)-equivariant map
\begin{equation*}
\pH(2)\setminus\pC(2)\ \to\ S^7\setminus\pC(2)\cong\Lambda^2_-\pC(2)
\end{equation*} 
that appeared in \cite{Atiyah-W:Mtheory}, see also
\cite{Miyaoka:BSmetric}. It is well known that the negative spinor
bundle over \( \pC(2) \) admits a complete metric with holonomy \(
\G_2 \), the so-called Bryant-Salamon metric \cite{Bryant-S:excep}.
By building on work of \cite{Apostolov-S:G2red, Gambioli-N-S:qK8d},
the authors have succeeded in identifying the \( 3 \)-form determining this
\( \G_2 \)-structure in terms of \( X \), the \( 4 \)-form and other
quaternionic data. A more complete study will appear in a forthcoming
paper.

\end{document}